\documentclass[11pt,reqno]{amsart}
\usepackage{color}

\usepackage{amsmath}
\usepackage{amsfonts,amscd}
\usepackage{amssymb}
\usepackage{url}
\usepackage{hyperref}

\usepackage[english]{babel}
\usepackage{booktabs}
\usepackage{tikz}

\theoremstyle{plain}
\newtheorem{theorem}                 {Theorem}      [section]

\newtheorem{lemma}        [theorem]  {Lemma}
\newtheorem{proposition}  [theorem]  {Proposition}

\theoremstyle{definition}

\newtheorem{definition}   [theorem]  {Definition}

\newtheorem{example}      [theorem]  {Example}

\newtheorem{remark}       [theorem]  {Remark}

\numberwithin{equation}{section}

\def \cn{{\mathbb C}}
\def \C{{\mathbb C}}
\def \hn{{\mathbb H}}
\def \H{{\mathbb H}}

\def \rn{{\mathbb R}}

\def \zn{{\mathbb Z}}

\def \B{\mathcal B}
\def \E{\mathcal E}
\def \F{\mathcal F}

\def \H{\mathcal H}

\def \P{\mathcal P}

\def \V{\mathcal V}

\def\nab#1#2{\hbox{$\nabla$\kern -.3em\lower 1.0 ex
		\hbox{$#1$}\kern -.1 em {$#2$}}}
\def\hatnab#1#2{\hbox{$\nabla$\kern -.3em\lower 1.0 ex
		\hbox{$#1$}\kern -.1 em {$#2$}}}

\def \Re{\mathfrak R\mathfrak e}

\def \GLR#1{\mathbf{GL}_{#1}(\rn)}

\def \glr#1{\mathfrak{gl}_{#1}(\rn)}
\def \GLC#1{\mathbf{GL}_{#1}(\cn)}
\def \glc#1{\mathfrak{gl}_{#1}(\cn)}
\def \GLH#1{\mathbf{GL}_{#1}(\hn)}

\def \SL2{\widetilde{\text{\bf SL}}_{2}(\rn)}

\def \SO#1{\mathbf{SO}(#1)}
\def \so#1{\mathfrak{so}(#1)}

\def \U#1{\text{\bf U}(#1)}
\def \u#1{\mathfrak{u}(#1)}

\def \SU#1{\text{\bf SU}(#1)}

\def \Sp#1{\text{\bf Sp}(#1)}
\def \sp#1{\mathfrak{sp}(#1)}

\DeclareMathOperator{\Div}{div} \DeclareMathOperator{\grad}{grad}

\DeclareMathOperator{\trace}{trace}

\numberwithin{equation}{section}
\allowdisplaybreaks

\begin{document}

	
	
\author{Sigmundur Gudmundsson}
\address{Mathematics, Faculty of Science\\
University of Lund\\
Box 118, Lund 221 00\\
Sweden}
\email{Sigmundur.Gudmundsson@math.lu.se}
	
\author{Thomas Jack Munn}
\address{Mathematics, Faculty of Science\\
University of Lund\\
Box 118, Lund 221 00\\
Sweden}
\email{Thomas.Munn@math.lu.se}

\title
[Minimal Submanifolds]
{Minimal Submanifolds via \\  Complex-Valued Eigenfunctions}

\begin{abstract}
In this work we introduce a new method for manufacturing minimal submanifolds of Riemannian manifolds of codimension two. For this we employ the so called {\it complex-valued eigenfunctions}. This is particularly interesting in the cases when the  Riemannian ambient manifold is compact.  We then give several explicit examples in important cases.
\end{abstract}
	

\maketitle

\section{Introduction}
\label{section-introduction}

The study of minimal submanifolds of a given ambient space plays a central role in differential geometry.  This has a long, interesting history and has attracted the interests of many profound mathematicians for many generations. 
The famous Weierstrass-Enneper representation formula 
$$\Phi(z)=\Re\int_{z_0}^z[F(w)(1-G^2(w),i\,(1+G^2(w)),2\,G(w))] dw,$$
for minimal surfaces in Euclidean $3$-space, brings complex analysis into play as a useful tool for the study of these beautiful objects.  Here $F:U\to\cn$ is a holomorphic function and $G:U\to\cn$ meromorphic, both defined on a simply connected open subset of the complex plane containing the point $z_0$. The resulting map $\Phi:U\to\rn^3$ is {\it harmonic} and {\it conformal}.  
For this classical problem we refer the reader to \cite{Die-Hil-Sau} and \cite{Mee-Per}.
\medskip

This was later generalised to the study of minimal surfaces in much more general ambient manifolds via {\it conformal}  {\it harmonic} maps.  The following result follows from the seminal paper \cite{Eel-Sam} of Eells and Sampson from 1964.  For this see also Proposition 3.5.1 of \cite{Bai-Woo-book}.

\begin{theorem}
Let $\phi:(M^m,g)\to (N,h)$ be a smooth conformal map between Riemannian manifolds.  If $m=2$ then $\phi$ is harmonic if and only if the image is minimal in $(N,h)$.
\end{theorem}

This result has turned out to be very useful in the construction of minimal surfaces in Riemannian symmetric spaces of various types.  For this we refer to \cite{Cal},
\cite{Eel-Woo}, \cite{Uhl}, \cite{Bur-Raw} and \cite{Bur-Gue}, just to name a few.
\medskip

In their work \cite{Bai-Eel} from 1981, Baird and Eells have shown that complex-valued harmonic morphisms from Riemannian manifolds are useful tools for the study of minimal submanifolds of codimension two.  This is due to the following special case of their famous result, here stated as Theorem \ref{theorem:Bai-Eel}. This can be seen as dual to the above mentioned generalisation of the Weierstrass-Enneper representation.

\begin{theorem}\cite{Bai-Eel}
\label{theorem:Bai-Eel-special}
Let $\phi:(M,g)\to\cn$ be a complex-valued harmonic morphism from a Riemannian manifold.  Then every regular fibre of $\phi$ is a minimal submanifold of $(M,g)$ of codimension {\it two}.
\end{theorem}

Complex-valued holomorphic functions $\phi:(M,g,J)\to\cn$ from K\" ahler manifolds are harmonic morphisms. It is a classical result that their regular fibres are minimal submanifolds of codimension two.
\smallskip

In \cite{Gud-Sak-1}, the authors introduce a general method for producing harmonic morphisms on Riemannian manifolds.  The ingredients for their recipe are the so called complex-valued eigenfamilies.  Explicit examples of such families have been constructed on all the classical Riemannian symmetric spaces, see \cite{Gud-Sak-1}, \cite{Gud-Sak-2}, \cite{Gud-Sif-Sob-2}, \cite{Gha-Gud-4} and \cite{Gha-Gud-5} and Table \ref{table-eigenfunctions}  below.  This leads to explicit constructions of a variety of minimal submanifolds of codimension two on all these important Riemannian manifolds.
\medskip

In this work we extend the investigation and provide a new scheme for constructing minimal submanifolds of codimension two.  The complex-valued eigenfunctions are the elements of the above mentioned families.  These are the building blocks for our new method described by our following main result.

\begin{theorem}\label{theorem-main}
Let $\phi:(M,g)\to\cn$ be a complex-valued eigenfunction on a Riemannian manifold, such that $0\in\phi(M)$ is a regular value for $\phi$.  Then the fibre $\phi^{-1}(\{0\})$ is a minimal submanifold of $M$ of codimension two.
\end{theorem}

The case when $(M,g)$ is a compact Riemannian symmetric space is of particular interest.  We construct multi-dimensional families of compact minimal submanifolds in several important cases.
\smallskip 

The only other general methods for constructing minimal submanifolds of codimension two, that the authors are aware of, can be found in \cite{Daj-Gro} and \cite{Daj-Vla}.  These are obtained by immersions rather that submersions. Hence they are different from those presented here.

	
\section{Eigenfunctions and Eigenfamilies}
\label{section-eigenfunctions}

Let $(M,g)$ be an $m$-dimensional Riemannian manifold and $T^{\cn}M$ be the complexification of the tangent bundle $TM$ of $M$. We extend the metric $g$ to a complex bilinear form on $T^{\cn}M$.  Then the gradient $\nabla\phi$ of a complex-valued function $\phi:(M,g)\to\cn$ is a section of $T^{\cn}M$.  In this situation, the well-known complex linear {\it Laplace-Beltrami operator} (alt. {\it tension field}) $\tau$ on $(M,g)$ acts locally on $\phi$ as follows
$$
\tau(\phi)=\Div (\nabla \phi)=\sum_{i,j=1}^m\frac{1}{\sqrt{|g|}} \frac{\partial}{\partial x_j}
\left(g^{ij}\, \sqrt{|g|}\, \frac{\partial \phi}{\partial x_i}\right).
$$
For two complex-valued functions $\phi,\psi:(M,g)\to\cn$ we have the following well-known fundamental relation
\begin{equation*}\label{equation-basic}
\tau(\phi\, \psi)=\tau(\phi)\,\psi +2\,\kappa(\phi,\psi)+\phi\,\tau(\psi),
\end{equation*}
where the complex bilinear {\it conformality operator} $\kappa$ is given by $$\kappa(\phi,\psi)=g(\nabla \phi,\nabla \psi).$$  Locally this satisfies 
$$\kappa(\phi,\psi)=\sum_{i,j=1}^mg^{ij}\cdot\frac{\partial\phi}{\partial x_i}\frac{\partial \psi}{\partial x_j}.$$

\begin{definition}\cite{Gud-Sak-1}\label{definition-eigenfamily}
Let $(M,g)$ be a Riemannian manifold. Then a complex-valued function $\phi:M\to\cn$ is said to be an {\it eigenfunction} if it is eigen both with respect to the Laplace-Beltrami operator $\tau$ and the conformality operator $\kappa$ i.e. there exist complex numbers $\lambda,\mu\in\cn$ such that $$\tau(\phi)=\lambda\cdot\phi\ \ \text{and}\ \ \kappa(\phi,\phi)=\mu\cdot \phi^2.$$	
A set $\E =\{\phi_i:M\to\cn\ |\ i\in I\}$ of complex-valued functions is said to be an {\it eigenfamily} on $M$ if there exist complex numbers $\lambda,\mu\in\cn$ such that for all $\phi,\psi\in\E$ we have 
$$\tau(\phi)=\lambda\cdot\phi\ \ \text{and}\ \ \kappa(\phi,\psi)=\mu\cdot \phi\,\psi.$$ 
\end{definition}
\medskip

For the standard odd-dimensional round spheres we have the following eigenfamilies based on the classical real-valued spherical harmonics.

\begin{example}\label{example-basic-sphere} 
Let $S^{2n-1}$ be the odd-dimensional unit sphere in the standard Euclidean space $\cn^{n}\cong\rn^{2n}$ and define $\phi_1,\dots,\phi_n:S^{2n-1}\to\cn$ by
$$\phi_j:(z_1,\dots,z_{n})\mapsto \frac{z_j}{\sqrt{|z_1|^2+\cdots +|z_n|^2}}.$$  Then the tension field $\tau$ and the conformality operator $\kappa$ on $S^{2n-1}$ satisfy
$$\tau(\phi_j)=-\,(2n-1)\cdot\phi_j\ \ \text{and}\ \ \kappa(\phi_j,\phi_k)=-\,1\cdot \phi_j\cdot\phi_k.$$
\end{example}
\medskip

For the standard complex projective space $\cn P^n$ we obtain a complex multi-dimensional eigenfamily, in a similar way.

\begin{example}\label{example-basic-projective-space}
Let $\cn P^n$ be the standard $n$-dimensional complex projective space. For a fixed integer $1\le\alpha < n+1$ and some $1\le j\le\alpha < k\le n+1$  define the function $\phi_{jk}:\cn P^n\to\cn$ by
$$\phi_{jk}:[z_1,\dots,z_{n+1}]\mapsto \frac
{z_j\cdot\bar z_k}{z_1\cdot \bar z_1+\cdots + z_{n+1}\cdot \bar z_{n+1}}.$$  
Then the tension field $\tau$ and the conformality operator $\kappa$ on $\cn P^n$ satisfy
$$\tau(\phi_{jk})=-\,4(n+1)\cdot\phi_{jk}\ \ \text{and}\ \ \kappa(\phi_{jk},\phi_{lm})=-\,4\cdot \phi_{jk}\cdot\phi_{lm}.$$
\end{example}
\medskip

The next result shows that a given eigenfamily $\E$ on $M$ induces a large collection $\P_d(\E)$ of such objects.

\begin{theorem}\cite{Gha-Gud-5}\label{theorem-polynomials}
Let $(M,g)$ be a Riemannian manifold and the set of complex-valued functions  $$\E=\{\phi_i:M\to\cn\,|\,i=1,2,\dots,n\}$$ 
be a finite eigenfamily i.e. there exist complex numbers $\lambda,\mu\in\cn$ such that for all $\phi,\psi\in\E$ $$\tau(\phi)=\lambda\cdot\phi\ \ \text{and}\ \ \kappa(\phi,\psi)=\mu\cdot\phi\,\psi.$$  
Then the set of complex homogeneous polynomials of degree $d$
	$$\P_d(\E)=\{P:M\to\cn\,|\, P\in\cn[\phi_1,\phi_2,\dots,\phi_n],\, P(\alpha\cdot\phi)=\alpha^d\cdot P(\phi),\, \alpha\in\cn\}$$ 
	is an eigenfamily on $M$ such that for all $P,Q\in\P_d(\E)$ we have
	$$\tau(P)=(d\,\lambda+d(d-1)\,\mu)\cdot P\ \ \text{and}\ \ \kappa(P,Q)=d^2\mu\cdot P\, Q.$$
\end{theorem}
\smallskip

\renewcommand{\arraystretch}{2}
\begin{table}[h]\label{table}
	\makebox[\textwidth][c]{
		\begin{tabular}{cccc}
			\midrule
			\midrule
$U/K$	& $\lambda$ & $\mu$ & Eigenfunctions \\
\midrule
\midrule
$\SO n$ & $-\,\frac{(n-1)}2$ & $-\,\frac 12$ & 
see \cite{Gud-Sak-1}\\
\midrule
$\SU n$ & $-\,\frac{n^2-1}n$ & $-\,\frac{n-1}n$ & 
see \cite{Gud-Sob-1} \\
\midrule
$\Sp n$ & $-\,\frac{2n+1}2$ & $-\,\frac 12$ & 
see \cite{Gud-Mon-Rat-1} \\
\midrule
$\SU n/\SO n$ & $-\,\frac{2(n^2+n-2)}{n}$& $-\,\frac{4(n-1)}{n}$ & 
see \cite{Gud-Sif-Sob-2} \\
\midrule
$\Sp n/\U n$ & $-\,2(n+1)$ & $-\,2$ & 
see \cite{Gud-Sif-Sob-2} \\
\midrule
$\SO{2n}/\U n$ & $-\,2(n-1)$ & $-1$ & 
see \cite{Gud-Sif-Sob-2} \\
\midrule
$\SU{2n}/\Sp n$ & $-\,\frac{2(2n^2-n-1)}{n}$ & $-\,\frac{2(n-1)}{n}$ & 
see \cite{Gud-Sif-Sob-2} \\
\midrule
$\SO{m+n}/\SO m\times\SO n$ & $-(m+n)$ & $-2$ & 
see \cite{Gha-Gud-4} \\
\midrule
$\U{m+n}/\U m\times\U n$ & $-2(m+n)$ & $-2$ & 
see \cite{Gha-Gud-5} \\
\midrule
$\Sp{m+n}/\Sp m\times\Sp n$ & $-2(m+n)$ & $-1$ & 
see \cite{Gha-Gud-5} \\
\midrule
\midrule
\end{tabular}	
}
\bigskip
\caption{Eigenfunctions on the classical compact  irreducible Riemannian symmetric spaces.}
\label{table-eigenfunctions}	
\end{table}
\renewcommand{\arraystretch}{1}

The much studied case of complex-valued eigenfunctions $\phi:(M,g)\to\cn$ with $(\lambda,\mu)=(0,0)$ is that of the {\it harmonic morphisms}.  For the general theory of harmonic morphisms between Riemannian
manifolds we refer to the excellent book \cite{Bai-Woo-book} and the regularly updated online bibliography \cite{Gud-bib}.
\smallskip

In the case of harmonic morphisms we have the following important result of Baird and Eells providing a nice geometric flavour to the theory.

\begin{theorem}\cite{Bai-Eel}\label{theorem:Bai-Eel}
Let $\phi:(M,g)\to (N^n,h)$ be a horizontally (weakly) conformal map between Riemannian manifolds. If
\begin{enumerate}
\item[i)] $n=2$, then $\phi$ is harmonic if and only if $\phi$ has minimal fibres at regular points,
\item[ii)] $n\ge 3$, then two of the following conditions imply the other
\begin{enumerate}
\item $\phi$ is a harmonic map,
\item $\phi$ has minimal fibres at regular points,
\item $\phi$ is horizontally homothetic.
\end{enumerate}
\end{enumerate}
\end{theorem}

In the spirit of Theorem \ref{theorem:Bai-Eel}, Baird and Gudmundsson later established the following result of Theorem \ref{theorem-Bai-Gud-1}.  This plays a central role in our  work.  In the next section we will explain the notions and background and provide a new proof avoiding the {\it stress-energy tensor} employed in \cite{Bai-Gud-1}.

\begin{theorem}\cite{Bai-Gud-1}\label{theorem-Bai-Gud-1}
Let $\phi: (M,g) \to (N^n,h)$ be a submersion and $K= \phi^{-1}(\{y\}) $ for some $y\in N$. If $\phi$ is horizontally conformal up to first order along $K$, then the following two conditions are equivalent
\begin{enumerate}
\item $\phi$ is $n$-harmonic along $K$ i.e. $\tau_n(\phi)(x) = 0$, for all $x \in K$,
\item $K$ is a minimal submanifold of $M$.
\end{enumerate}
\end{theorem}

\section{Horizontal Conformality up to First Order}

From now on, we will turn our focus to the complex-valued eigenfunctions $\phi:(M,g)\to\cn$ with $(\lambda,\mu)\neq(0,0)$. The existence of these objects has been proven in several important cases, for example when $(M,g)$ is a classical Riemannian symmetric space, see the articles \cite{Gud-Sak-1}, \cite{Gud-Sak-2}, \cite{Gud-Sif-Sob-2}, \cite{Gha-Gud-4} and \cite{Gha-Gud-5}.
\smallskip

\begin{definition}
Let $\phi:(M,g) \to (N,h)$ be a smooth map between Riemannian manifolds. For a real number $p> 1$, we say that $\phi$ is a \emph{$p$-harmonic map} if it is a critical point of the $p$-energy functional,
$$ E_p(\phi) = \frac{1}{p} \int_M |d \phi|^p dx.$$
\end{definition}

The $p$-harmonic maps are characterised by the following result.

\begin{proposition} \cite{Bai-Gud-1}
A smooth map $\phi: (M,g) \to (N,h)$ between Riemannian manifolds is $p$-harmonic if and only if it satisfies the Euler-Lagrange equation
$$ \tau_p(\phi) = |d\phi|^{p-2} \cdot [\tau(\phi)+d\phi(\grad(\log|d\phi|^{p-2}))] = 0 .$$
We call $\tau_p (\phi)$ the $p$-tension field of $\phi$.
\end{proposition}

Notice that $p$-harmonic maps generalise the classical case of harmonic maps, since when $p=2$, the $p$-energy functional is the standard energy functional.

Our next item is the definition of an important ingredient for Theorem \ref{theorem-Bai-Gud-1}.

\begin{definition}\cite{Bai-Gud-1}
Let $\phi: (M,g)\to (N^n,h)$ be a smooth submersion between Riemannian manifolds.  Further  let the functions $\lambda_1^2,\dots,\lambda_n^2$ denote the non-zero eigenvalues of the first fundamental form $\phi^* h$ with respect to the metric $g$. Let $K$ be a submanifold of $M$. Then $\phi$ is said to be {\it horizontally conformal up to first order} along $K$ if
\begin{enumerate}
\item[(i)]  $\lambda_1^2(p)= \cdots=\lambda_n^2(p)$ and
\item[(ii)] $(\nabla\lambda_1^2)(p) =\cdots=(\nabla\lambda_n^2)(p)$, at every point $p \in K$.
\end{enumerate}
\end{definition}
	
\begin{definition}
Let $ \hat X$ be a horizontal vector field over $M$. Then $\hat X$ is said to be \emph{basic} if there exists a vector field $ X$ on $N$ such that $d\phi_x( \hat X) = X$ for all $x \in M.$ The vector field $\hat X$ is called a \emph{horizontal lift} of $X$. For a given vector field $X$ on $N$ there exists a unique horizontal lift $\hat X$ on $M$. For two basic horizontal vector fields $\hat X, \hat Y$ on $M$ we write $\widehat{\nab{X}{Y}}$ for the horizontal lift of $ \nab{X}{Y}$, where $\nabla$ is the Levi-Civita connection on $N$.
\end{definition}

By $\H$ we also denote the projection of $TM$ onto the horizontal space $\H$.
	
\begin{lemma}\cite{Gud-PhD-thesis}\label{Baird-Gud 3.1}
If $\phi:(M,g) \to (N,h)$ is a horizontally conformal submersion and $\hat X, \hat Y$ are basic vector fields on $M$, then
$$\mathcal{H} \hatnab{ \hat X}{ \hat Y} = \widehat{\nab{ X}{ Y}} + \frac{\lambda^2}{2} \left\{ \hat X(\frac{1}{\lambda^2})\cdot \hat Y+ \hat Y(\frac{1}{\lambda^2})\cdot \hat X -g(\hat X, \hat Y)\cdot \grad_{\mathcal{H}} (\frac{1}{\lambda^2})\right\}. $$
\end{lemma}
	
This has the following interesting consequence.

\begin{proposition}
\cite{Bai-Gud-1} \label{Baird-Gud 3.2}
Let $\phi:(M,g) \to (N,h)$ be a smooth  submersion and $\hat X,\hat Y$ be basic vector fields on $M$. If $K$ is a fibre of $\phi$, hence a submanifold of $M$, such that $\phi$ is horizontally conformal up to first order along $K$, then
$$\mathcal{H} \hatnab{ \hat X}{ \hat Y} = \widehat{\nab{ X}{ Y}} + \frac{\lambda^2}{2} \left\{ \hat X(\frac{1}{\lambda^2})\cdot \hat Y+ \hat Y(\frac{1}{\lambda^2})\cdot \hat X -g(\hat X, \hat Y)\cdot \grad_{\mathcal{H}} (\frac{1}{\lambda^2})\right\} $$
along $K$.
\end{proposition}

\begin{proof}
Since the proof in \cite{Gud-PhD-thesis} only uses the value of $\lambda$ and its first derivative we immediately obtain the result.	
\end{proof}

\begin{proof}[Proof of Theorem \ref{theorem-Bai-Gud-1}]
Let $\{Z_1,..., Z_n \}$ be a local orthonormal frame for the tangent bundle $TN$ of $N$ and lift horizontally to $\hat Z_i$ on $K$, then $\{\lambda \hat Z_i \mid i=1,...,n\}$ is an orthonormal frame for the horizontal distribution on $K$. Recalling the formula
$$ \tau_p(\phi) = |d\phi|^{p-2} \cdot [\tau(\phi)+d\phi(\grad(\log|d\phi|^{p-2}))],$$
we first calculate the horizontal component of the Laplace-Beltrami operator.	
\begin{eqnarray}\label{equation-horizontal-Laplacian}
\tau_{\mathcal{H}}(\phi) 
&=& \sum_{k=1}^n \{ \nabla^{\phi^* TN}{}_{\lambda \hat Z_k} d\phi(\lambda \hat Z_k)- d\phi(\hatnab{\lambda \hat Z_k}{\lambda \hat Z_k)} \}\nonumber \\
&=& \sum_{k=1}^n \{ \lambda \hat Z_k(\lambda) \cdot d\phi(\hat Z_k)+\lambda^2 \nabla^{\phi^* TN}{}_{\hat Z_k} d\phi(\hat Z_k)\nonumber \\ 
&& \qquad - \lambda \hat Z_k(\lambda) \cdot d\phi(\hat Z_k) - \lambda^2d\phi(\hatnab{\hat Z_k}{\hat Z_k}) \}\nonumber\\
&=&\lambda^2 \sum_{k=1}^n\{ \nab{Z_k}{Z_k} - d\phi(\hatnab{\hat Z_k}{\hat Z_k)} \}.
\end{eqnarray}

Then since $\phi$ is horizontally conformal up to first order along $K$, it follows from Proposition \ref{Baird-Gud 3.2} that
\small{$$\mathcal{H} \hatnab{ \hat Z_k}{ \hat Z_k} = \widehat{\nab{ Z_k}{ Z_k}} + \frac{\lambda^2}{2} \left\{ \hat Z_k(\frac{1}{\lambda^2})\cdot \hat Z_k+ \hat Z_k(\frac{1}{\lambda^2})\cdot \hat Z_k -g(\hat Z_k, \hat Z_k)\cdot \grad_{\mathcal{H}} (\frac{1}{\lambda^2})\right\}$$}
along the submanifold $K$ i.e.
$$\widehat{\nab{ Z_k}{ Z_k}}-\mathcal {H} \hatnab{\hat Z_k}{ \hat Z_k} =  -\lambda^2 {\hat Z_k}(\frac{1}{\lambda^2})\cdot \hat Z_k+ \frac{1}{2}\cdot \grad_{\mathcal{H}} (\frac{1}{\lambda^2}). $$
Then projecting onto $TN$ by $d\phi$, we obtain
\begin{eqnarray*}
\nab{Z_k}{Z_k}- d\phi( \hatnab{\hat Z_k}{\hat Z_k}) =  -d\phi(\lambda {\hat Z_k}(\frac{1}{\lambda^2})\cdot \lambda \hat Z_k)+ \frac{1}{2}\cdot d\phi( \grad_{\mathcal{H}} (\frac{1}{\lambda^2})).
\end{eqnarray*}
So we can simplify the above equation (\ref{equation-horizontal-Laplacian}) and obtain
\begin{eqnarray*}
\tau_{\mathcal{H}}(\phi) 
&=& \lambda^2 \cdot \sum_{k=1}^n\{ -d\phi(\lambda {\hat Z_k}(\frac{1}{\lambda^2})\cdot \lambda \hat Z_k)+ d\phi(\frac{1}{2}\cdot \grad_{\mathcal{H}} (\frac{1}{\lambda^2})) \}\\
&=& \lambda^2 \cdot  \{-d\phi(\grad_{\mathcal{H}}(\frac{1}{\lambda^2})) + \frac{n}{2} \cdot d\phi(\grad_{\mathcal{H}}(\frac{1}{\lambda^2})) \}\\
&=& \lambda^2 \cdot \frac{n-2}{2} \cdot d \phi (\grad_{\mathcal{H}}(\frac{1}{\lambda^2})).
\end{eqnarray*}

Let $\{V_{n+1},...,V_m\}$ be a local orthonormal frame for the vertical distribution. Then for the tension field $\tau(\phi)$ along $K$ we obtain
\begin{eqnarray*}
\tau(\phi) 
&=& \tau_{\mathcal{H}}(\phi) + \sum_{k=n+1}^m g(V_k,V_k) \cdot \{\nabla^{\phi^* TN}{}_{V_k} d\phi(V_k) - d\phi( \hatnab{V_k}{V_k}) \}\\
&=& \lambda^2 \cdot   \frac{n-2}{2} \cdot d \phi (\grad_{\mathcal{H}}(\frac{1}{\lambda^2})) - \sum_{k=n+1}^m d\phi(  \hatnab{V_k}{V_k} )\\
&=&  \lambda^2 \cdot \frac{n-2}{2} \cdot d \phi (\grad_{\mathcal{H}}(\frac{1}{\lambda^2})) -(m-n) \cdot d\phi(H).
\end{eqnarray*}

Now we consider $d\phi(\grad(\log|d\phi|^{p-2}))$,
\begin{eqnarray*}
d\phi(\grad(\log|d\phi|^{p-2})) 
&=& \frac{p-2}{2} \cdot   d\phi(\grad(\log|d\phi|^{2}))\\
&=&\frac{p-2}{2} \cdot d\phi(\frac{1}{|d\phi|^2} \grad(|d\phi|^2)).
\end{eqnarray*}
Also, because $d\phi(V) = 0$ for all $V \in \mathcal{V}$ we obtain
\begin{eqnarray*}
|d\phi|^2
&=& \sum_{i=1}^n h(d\phi (\lambda \hat{Z}_i), d\phi (\lambda \hat{Z}_i))\\
&=& \sum_{i=1}^n h(\lambda {Z}_i, \lambda {Z}_i)\\
&=& n\cdot \lambda^2.
\end{eqnarray*}
		
Then for $n=p$ we have that
\begin{eqnarray*}
\tau_n(\phi) 
&=& |d\phi|^{n-2} \cdot \Big[ \lambda^2 \cdot \frac{n-2}{2} \cdot  d\phi(\grad_{\mathcal{H}}(\frac{1}{\lambda^2}))-(m-n) \cdot d\phi(H)\\
&& \qquad \qquad \qquad \qquad +\frac{n-2}{2} \cdot d\phi(\frac{1}{\lambda^2} \cdot \grad_{\mathcal{H}}(\lambda^2))\Big]\\
&=&  |d\phi|^{n-2} \cdot  \Big[ \frac{n-2}{2} \cdot d\phi\Big( \lambda^2 \grad_{\mathcal{H}}(\frac{1}{\lambda^2})+\frac{1}{\lambda^2} \grad_{\mathcal{H}}(\lambda^2)\Big)\\ 
&&\qquad \qquad \qquad \qquad \qquad  -(m-n) \cdot d\phi(H)\Big]\\
&=&|d\phi|^{n-2} \cdot  \Big[ \frac{n-2}{2} \cdot d\phi \Big( \grad_{\mathcal{H}} \log(\frac{1}{\lambda^2}) + \grad_{\mathcal{H}} \log ({\lambda^2})\Big)\\ 
&& \qquad \qquad \qquad \qquad \qquad-(m-n) \cdot d\phi(H)\Big]\\
&=& -|d\phi|^{n-2}\cdot  (m-n) \cdot d\phi(H).
\end{eqnarray*}
This shows that $\phi$ is $n$-harmonic if and only if $d\phi(H) = 0$ along the submanifold $K$ of $M$. 
\end{proof}

\section{Minimality of the fibre over the origin}

In this section we prove our main result stated in Theorem \ref{theorem-main}.  We equip the complex plane $\C$ with its standard Euclidean metric denoted by $h$.

\begin{lemma}\label{lemma-eigenvalues}
Let $\phi = u+iv:(M,g)\to (\cn,h)$ be a smooth submersion that is eigen with respect to the conformality operator $\kappa$ i.e. $\kappa (\phi,\phi )=\mu\cdot \phi^2$.  Then the two eigenvalues $\lambda_1$ and $\lambda_2$ of the first fundamental form $\phi^*h$ satisfy
$$\lambda_{1,2}=\tfrac 12\cdot \big( (|\nabla u|^2+|\nabla v|^2)\pm\mu\cdot (u^2+v^2)\big).$$
\end{lemma}

\begin{proof}
It is an immediate consequence of the fact that the function $\phi$ is eigen with respect to the conformality operator $\kappa$ that 
$$\mu(u^2-v^2) = |\nabla u|^2 - |\nabla v|^2 \text{ and }\mu\cdot uv = g(\nabla u,\nabla v).$$
This will lead to important simplifications in the following calculations.
\medskip

We first notice that for any tangent vector $X$ of $M$ we have that
$$d\phi(X)= du(X)+i\cdot dv(X)= g(\nabla u,X)+i\cdot g(\nabla v,X).$$
This implies that for any vertical vector $V\in\V$ tangent to a fibre of $\phi$ we have $d\phi(V)=0$. Furthermore, at each point of $M$ the vectors $\nabla u$ and $\nabla v$ form a basis for the horizontal space $\H$. For these we have
$$ d\phi(\nabla u) = g(\nabla u,\nabla u) + i\cdot g(\nabla v,\nabla u),$$
$$d\phi( \nabla v) = g(\nabla u,\nabla v) +i\cdot g(\nabla v,\nabla v).$$

We can now describe the first fundamental form $\phi^* h$ in terms of the basis $\{\nabla u,\nabla v\}$ of the horizontal distribution $\H$ as follows.
\begin{eqnarray*}
&&h(d \phi(\nabla u), d\phi(\nabla u))\\
&=&h( g(\nabla u,\nabla u) + i\cdot g(\nabla v,\nabla u),  g(\nabla u,\nabla u) + i\cdot g(\nabla v,\nabla u))\\
&=& g(\nabla u,\nabla u)^2+g(\nabla v,\nabla u)^2\\
&=& |\nabla u|^4 + \mu^2 \cdot u^2 v^2,
\end{eqnarray*}
\begin{eqnarray*}
h(d\phi(\nabla u), d \phi(\nabla v))
&=& \mu \cdot u v(|\nabla u|^2+|\nabla v|^2),
\end{eqnarray*}
\begin{eqnarray*}
h( d\phi(\nabla v), d \phi(\nabla v) )
&=& |\nabla v|^4 + \mu^2 \cdot u^2 v^2.
\end{eqnarray*}

Using the Gram-Schmidt process we now yield the following orthonormal basis $\{N_1,N_2\}$ for the horizontal space at each point
$$ N_1 = \frac{\nabla u}{|\nabla u|}, \ \ N_2 = \frac{{|\nabla u|^2}\nabla v - (\mu \cdot uv) \cdot \nabla u}{|{|\nabla u|^2}\nabla v -(\mu \cdot uv) \cdot \nabla u|}. $$ 
The linearity of the Euclidean metric $h$ on $\cn$ then leads to 
$$h( d\phi(N_1), d \phi(N_1) ) = |\nabla u|^2 +  \frac{\mu^2 \cdot u^2 v^2}{|\nabla u|^2},$$
\begin{eqnarray*}
h ( d\phi(N_1), d \phi(N_2) )
&=& \frac{(\mu\cdot uv) \cdot (|\nabla u |^2 |\nabla v|^2 - \mu^2 \cdot u^2 v^2)}{{|\nabla u| \cdot | \,|\nabla u|^2}\nabla v -(\mu \cdot uv) \cdot \nabla u|}
\end{eqnarray*}
and 
\begin{eqnarray*}
h ( d\phi(N_2), d \phi(N_2))
&=&  \frac{(|\nabla u|^2 \cdot|\nabla v|^2 - \mu^2 \cdot u^2 v^2)}{|\nabla u|^2}.
\end{eqnarray*}

We are now in the position of calculating the eigenvalues of the first fundamental form $\phi^*h$ with respect to the metric $g$ on $M$.  These are the the solutions to the characteristic polynomial $\det(\lambda\cdot I-A)=0$, where $A$ is the real $(2\times 2)$-matrix
$$A=\begin{bmatrix}
h ( d\phi(N_1), d \phi(N_1) ) & h ( d\phi(N_1), d \phi(N_2) )\\
h( d\phi(N_2), d \phi(N_1) )&  h( d\phi(N_2), d \phi(N_2) )
\end{bmatrix}.$$
This simplifies to the following second order polynomial equation in $\lambda$
\begin{eqnarray*}
0&=&a\cdot \lambda^2+b\cdot \lambda+c\\
&=&\lambda^2- (|\nabla u|^2+|\nabla v|^2)\cdot \lambda+(|\nabla u|^2\cdot |\nabla v|^2 - \mu^2 \cdot u^2 v^2).
\end{eqnarray*}
Now considering the discriminant $D=b^2-4ac$ we then, after simplifying, obtain  
\begin{eqnarray*}
D&=& {\mu^2}\cdot (u^2+v^2)^2.
\end{eqnarray*}
With this at hand, we finally obtain the two solutions $\lambda_1$ and $\lambda_2$ to the characteristic polynomial as
\begin{eqnarray*}
&&\lambda_{1,2} = \frac{(|\nabla u|^2+|\nabla v|^2)\pm  \mu \cdot (u^2+v^2) }{2}.
\end{eqnarray*}
\end{proof}

We now prove the main result of this work formulated in Theorem \ref{theorem-main}.

\begin{proof}[Proof of Theorem \ref{theorem-main}]
Let $p\in\F_0$ be a point of the fibre of $\phi=u+i\cdot v$ over the origin  $(u,v)=0\in\cn$.  Lemma \ref{lemma-eigenvalues} implies that $\lambda_1(p)=\lambda_2(p)$ and if $X\in T_pM$ is a tangent vector of $M$ at $p$,  then
\begin{eqnarray*}
X(\lambda_1^2-\lambda_2^2)(p)
&=&X\Big( \mu\cdot (|\nabla u|^2+|\nabla v|^2)\cdot (u^2+v^2)\Big)(p)\\
&=&\mu\cdot X(|\nabla u|^2+|\nabla v|^2)(p)\cdot (u^2+v^2)(p)\\
&&+\,\mu\cdot (|\nabla u|^2+|\nabla v|^2)(p)\cdot X(u^2+v^2)(p)\\
&=&2\,\mu\cdot (|\nabla u|^2+|\nabla v|^2)(p) (X(u)\cdot u+X(v)\cdot v)(p)\\
&=&0.
\end{eqnarray*}
This shows that the complex-valued function $\phi$ is horizontally conformal up to first order along $\F_0$.  Following Theorem \ref{theorem-Bai-Gud-1}, the fibre $\F_0$ is a minimal submanifold of $(M,g)$.
\end{proof}

\section{Simple Compact Examples}

In this section we employ Theorem \ref{theorem-main} to construct compact minimal submanifolds.  Let $\Phi:(M,g)\to\cn$ be a complex-valued eigenfunction on a compact Riemannian manifold $M$ and $0\in\cn$ be a regular value of $\Phi$.  Then the fibre $\Phi^{-1}(\{0\})$ is closed and hence compact as a submanifold of $M$.
\medskip

Let $\cn^{n}\cong\rn^{2n}$ be the standard $2n$-dimensional Euclidean space and $d\in\zn^+$ be a positive integer.  Further let $P:\cn^n\to\cn$ be a homogeneous polynomial of degree $d$ i.e. $P(r\cdot z)=r^d\cdot P(z)$ for all $r\in\rn^+$.  Now define the function $\hat\Phi:\cn^n\setminus\{0\}\to\cn$ by  $\hat\Phi (z)=P(z)/{|z|^d}$.  Then $\hat\Phi(r\cdot z)=\hat\Phi(z)$ for all $r\in\rn^+$ and $z\in\cn^n\setminus\{0\}$ so $\hat\Phi$ induces a complex-valued function $\Phi$ on the odd-dimensional unit sphere $S^{2n-1}$ in $\cn^n$ which simply is the restriction of $\hat\Phi$ to the sphere.  It is clear that for any point $p\in S^{2n-1}$ the two gradients $\nabla\hat\Phi$ and $\nabla\Phi$ satisfy $$\nabla\Phi(p)=\nabla\hat\Phi(p).$$ 
For the partial derivative $\partial\hat\Phi/\partial z_k$ and $\partial\hat\Phi/\partial \bar z_k$ we now have 
\begin{equation}
\frac{\partial\hat\Phi}{\partial z_k}(z)= \frac{
\frac{\partial P}{\partial z_k}(z)\cdot |z|^d
-\bar z_k\,|z|^{d-1}\cdot P(z)}{|z|^{2d}}
\end{equation}
and
\begin{equation}
\frac{\partial\hat\Phi}{\partial \bar z_k}(z)= \frac{
\frac{\partial P}{\partial \bar z_k}(z)\cdot |z|^d
-z_k\,|z|^{d-1}\cdot P(z)}{|z|^{2d}}.
\end{equation}

After our preparations we are now ready for a few explicit examples.  We start with a complex multi-dimensional family of compact minimal submanifolds on the odd-dimensional spheres.

\begin{example}
Let $S^{2n-1}$ be the odd-dimensional unit sphere in the standard Euclidean $\cn^n\cong\rn^{2n}$ and $A\in\cn^{n\times n}$ be a complex matrix which is invertible i.e. $\det A\neq 0$.  Further define the function $\Phi_A:S^{2n-1}\to\cn$ with
$$\Phi_A:z=(z_1,z_2,\dots,z_n)\mapsto \frac 1{|z|^2}\cdot\Big(\sum_{k\neq l}^na_{kl}\, z_kz_l+\tfrac 12\cdot\sum_{k=1}^na_{kk}\,z_k^2\Big).$$
If $p\in S^{2n-1}$ is a element of the unit sphere then for $k=1,2,\dots,n$ we have 
\begin{equation}\label{equation-sphere-1}
	\frac{\partial\hat\Phi_A}{\partial z_k}(p)= \frac{\partial P_A}{\partial z_k}(p)
	-\bar p_k\cdot P_A(p)
\end{equation}
and
\begin{equation}\label{equation-sphere-2}
	\frac{\partial\hat\Phi_A}{\partial \bar z_k}(p)= 
	\frac{\partial P_A}{\partial \bar z_k}(p)
	-p_k\cdot P_A(p)=-p_k\cdot P_A(p).
\end{equation}

Let us now assume that $p\in S^{2n-1}$ is a critical point of $\Phi_A$.  If $p$ is not in the fibre $\Phi_A^{-1}(\{0\})$ then  equation (\ref{equation-sphere-2}) shows that $p=0\notin S^{2n-1}$.  On the other hand if the critical point $p\in\Phi_A^{-1}(\{0\})$ then we see from equation (\ref{equation-sphere-1}) that for $k=1,2,\dots ,n$, we have 
$$\frac{\partial \hat\Phi_A}{\partial z_k}(p)=\frac{\partial P_A}{\partial z_k}(p)=a_{k1}\, p_1+\cdots +a_{kn}\, p_n.$$
This shows that $p$ is a critical point for $\Phi_A$ if and only if $A\cdot p=0$.  Since $\det A\neq 0$ this implies that $p=0\notin S^{2n-1}$. 

These considerations show that the compact fibres of the function $\Phi_A:S^{2n-1}\to\cn$ form a foliation on the unit sphere of codimension two. From Example \ref{example-basic-sphere} and Theorem \ref{theorem-polynomials}, $\Phi_A$ is an eigenfunction on $S^{2n-1}$.  According to Theorem \ref{theorem-main} the compact submanifold $\Phi_A^{-1}(\{0\})$ is minimal in $S^{2n-1}$.  This means that we have constructed a complex $n^2$-dimensional family $\F_n$ of compact minimal submanifolds of codimension two, given by 
$$\F_n=\{\Phi_A^{-1}(\{0\})\,|\, A\in\cn^{n\times n},\ \det A\neq 0\}.$$\end{example}
\smallskip

With the next example we provide a complex multi-dimensional family of minimal submanifolds of the odd-dimensional complex projective spaces $\cn P^{2n-1}$.

\begin{example}\label{example-projective-space}
For a positive integer $n\in \zn^+$, let $\Sigma_n$ be the following subset of $\cn^n$ given by 
$$\Sigma_n=\{(a_1,a_2,\dots ,a_n)\in\cn^n\, | \, a_k\neq 0\ \text{for}\  k=1,2,\dots,n\}.$$
It then follows from Theorem \ref{theorem-polynomials} and Example \ref{example-basic-projective-space} that for all $a\in\Sigma_n$ the complex-valued function $\Phi_a:\cn P^{2n-1}\to\cn$ with 
$$\Phi_a:[z_1,z_2,\dots,z_{2n}]\mapsto \frac 1{|z|^2}
\cdot\big( a_1\cdot z_1\bar z_{n+1}+\dots +a_n\cdot z_n\bar z_{2n}\big) .$$
is an eigenfunction on the complex projective space $\cn P^{2n-1}$. This clearly lifts to $\hat\Phi_a=\pi\circ\Phi_a$ on the $(4n-1)$-dimensional unit sphere $S^{4n-1}$ in $\cn^{2n}\cong\rn^{4n}$, via the standard submersive Hopf fibration $\pi:S^{4n-1}\to\cn P^{2n-1}$.  Then we can use the machinery developed earlier in this section for the lift $\hat\Phi$.

If $p\in S^{4n-1}$ is a element of the unit sphere then for $k=1,2,\dots,n$ we have 
\begin{equation}\label{equation-sphere-3}
\frac{\partial\hat\Phi_a}{\partial z_k}(p)= \frac{\partial P_a}{\partial z_k}(p)	-\bar p_{k}\cdot P_a(p)
=a_k\cdot\bar p_{n+k}-\bar p_{k}\cdot P_a(p)
\end{equation}
and
\begin{equation}\label{equation-sphere-4}
\frac{\partial\hat\Phi_a}{\partial \bar z_{n+k}}(p)= 
\frac{\partial P_a}{\partial \bar z_{n+k}}(p)
-p_{n+k}\cdot P_a(p)=a_k\cdot p_k-p_{n+k}\cdot P_a(p).
\end{equation}


For a point $p\in S^{4n-1}$ contained in the compact fibre $\hat\Phi_a^{-1}(\{0\})$ we then have, for $k=1,2,\dots ,n$, that 
$$\frac{\partial \hat\Phi_a}{\partial z_k}(p)=\frac{\partial P_a}{\partial z_k}(p)=a_k\cdot\bar p_{n+k}\ \ \text{and}\ \ 
\frac{\partial\hat \Phi_a}{\partial \bar z_{n+k}}(p)=\frac{\partial P_a}{\partial \bar z_{n+k}}(p)=a_k\cdot p_k.$$
This implies that $\nabla\Phi_a(p)\neq 0$ along the fibre $\Phi_a^{-1}(\{0\})$ over the origin $0\in\cn$.  According to the implicit function theorem this fibre is a submanifold of $\cn P^{2n-1}$ of codimension two.  Furthermore Theorem \ref{theorem-main} tells us that it is minimal.  Here we have constructed a complex $n$-dimensional family $\F_n$ of minimal submanifolds of $\cn P^{2n-1}$ satisfying
$$\F_n=\{\Phi_a^{-1}(\{0\})\,|\, a\in\Sigma_n\}.$$
\end{example}

\section{The Riemannian Lie group $\GLC n$}

The main purpose of this short section is to introduce some useful notation.  The complex general linear group is given by 
$$\GLC{n}=\{z\in\cn^{n\times n}\ |\ \det z \neq 0\}.$$
Its Lie algebra $\glc n$ of left-invariant vector fields can be identified with the tangent space at the neutral element $e\in\GLC n$ i.e. the $n\times n$ complex matrices in $\cn^{n\times n}$.  We equip $\GLC n$ with its standard Riemannian metric induced by the Euclidean scalar product
on the Lie algebra $\glc n$ given by
$$g(Z,W)=\Re\trace (ZW^*).$$  
For $1\le i,j\le n$ we shall by $E_{ij}$ denote the element of $\glr n$ satisfying
$$(E_{ij})_{kl}=\delta_{ik}\delta_{jl}$$ 
and by $D_t$ the diagonal matrices $D_t=E_{tt}$. For $1\le r<s\le n$ let $X_{rs}$ and $Y_{rs}$ be the matrices satisfying
$$X_{rs}=\frac 1{\sqrt 2}(E_{rs}+E_{sr}),\ \ Y_{rs}=\frac
1{\sqrt 2}(E_{rs}-E_{sr}).$$ 

\section{The compact special orthogonal group $\SO n$}

In this section we construct minimal submanifolds of codimension two of the compact special orthogonal group
$$\SO{n}=\{x\in\GLR{n}\ |\ x\cdot x^t=I_n,\ \det x =1\}.$$
For this we use its standard $n$-dimensional representation on $\cn^n$ with 
$$x\mapsto
\begin{bmatrix}
x_{11} & \cdots & x_{1n} \\ 
\vdots & \ddots & \vdots \\
x_{n1} & \cdots & x_{nn}
\end{bmatrix}.$$ 
The Lie algebra $\so n$ of $\SO n$ is the set of real  skew-symmetric matrices
$$\so n=\{X\in\glr n\ |\ X+X^t=0\}$$ 
and for this we have the canonical orthonormal basis
$$\B_{\so n}=\{Y_{rs}\ |\ 1\le r<s\le n\}.$$
The gradient $\nabla\phi$, of a complex-valued function $\phi:\SO n\to\cn$, is an element of the complexified tangent bundle $T^\cn\SO n$.  This satisfies 
$$\nabla\phi=\sum_{Y\in\B_{\so n}}Y(\phi)\cdot Y.$$
We now consider the functions $x_{j\alpha}: \SO{n} \to \cn$ with $x_{j\alpha} : x \mapsto e_j \cdot x \cdot e_\alpha^t$. For any tangent vector $Y\in B_{\u{n}}$ we have 
$$Y(x_{j\alpha}) : x \mapsto e_j \cdot x \cdot Y  \cdot e_\alpha^t = \sum_{k=1}^n x_{jk}\cdot Y_{k\alpha}.$$ 
This implies that 
\begin{equation}\label{equation-basic-SO(n)}
Y_{rs}(x_{j \alpha}) = \frac{1}{\sqrt{2}}\, (  
x_{jr}\cdot \delta_{\alpha s}
-x_{js}\cdot \delta_{\alpha r}).
\end{equation}

For the tension field $\tau$ and the conformality operator $\kappa$ on the special orthonormal group $\SO n$ we have the following result.

\begin{lemma}\cite{Gud-Sak-1}
For $1\le j,\alpha\le n$, let $x_{j\alpha}:\SO {n}\to\rn$ be the real-valued matrix elements of the standard representation of $\SO {n}$.  Then the following relations hold  
$$\tau(x_{j\alpha})=-\,\frac {(n-1)}2\cdot x_{j\alpha},$$
$$\kappa(x_{j\alpha},x_{k\beta})=-\,\frac 12\cdot (x_{j\beta}x_{k\alpha}-\delta_{jk}\delta_{\alpha\beta}).$$
\end{lemma}

This leads to the following statement.

\begin{proposition}\cite{Gud-Sak-1}\label{proposition-SO(n)-eigen}
Let $p\in\cn^n$ be a non-zero element and $V$ be a maximal
isotropic subspace of $\cn^n$.  Then the complex vector space
$$\E_V(p)=\{\phi_a:\SO n\to\cn\ |\ \phi_a(x)=\trace (p^tax^t),\ a\in V\}$$ is an eigenfamily on $\SO n$.
\end{proposition}

\begin{example}
Let us consider the real-valued functions  $x_{11},x_{12}: \SO{n} \to \cn$ with $x_{j\alpha} : x \mapsto e_j \cdot x \cdot e_\alpha^t$. For a tangent vector $Y_{rs}\in B_{\so{n}}$ the equation (\ref{equation-basic-SO(n)}) 
implies that 
$$
Y_{12}(x_{11})=-\frac {x_{12}}{\sqrt{2}},\ 
Y_{13}(x_{11})=-\frac {x_{13}}{\sqrt{2}},\ \dots\ ,\ 
Y_{1n}(x_{11})=-\frac {x_{1n}}{\sqrt{2}},
$$
$$
Y_{12}(x_{12})=\frac {x_{11}}{\sqrt{2}},\ 
Y_{23}(x_{12})=-\frac {x_{13}}{\sqrt{2}},\ \dots\ ,\ 
Y_{2n}(x_{12})=-\frac {x_{1n}}{\sqrt{2}}.
$$
Then define the complex-valued function $\Phi:\SO n\to\cn$ with $\Phi(x)=(x_{11}+i\, x_{12})$.  The above derivatives show that the gradient $$\nabla\Phi=\nabla x_{11}+i\,\nabla x_{12}$$ never vanishes along $\SO n$, so $\Phi$ induces a foliation on $\SO n$ of codimension two.  According to Proposition \ref{proposition-SO(n)-eigen}, $\Phi$ is an eigenfunction. Hence the fibre $\Phi^{-1}(\{0\})$ is a compact minimal submanifold of $\SO {n}$, by Theorem \ref{theorem-main}.
\end{example}

\begin{example}\label{example-d-SO(n)}
Let $d\in\zn^+$ be a positive integer and $a_1,\dots ,a_n\in\cn$ be non-zero complex numbers.  Now we consider the complex-valued function $\Phi:\SO{2n}\to\cn$ with 
$$\Phi(x)=\sum_{k=1}^n a_k\, (x_{1,2k-1}+i\, x_{1,2k})^d.$$
Then it easily follows from the chain rule and equation (\ref{equation-basic-SO(n)}) that for $k=1,2,\dots, n$ we have 
$$Y_{2k-1,2k}(\Phi)=\frac{i\, a_k\, d}{\sqrt 2}(x_{1,2k-1}+i\, x_{1,2k})^d.$$
This shows that the gradient $\nabla\Phi$ of $\Phi$ is non-vanishing along $\SO {2n}$, so $\Phi$ induces a foliation on $\SO {2n}$ of codimension two.  According to Proposition \ref{proposition-SO(n)-eigen} and Theorem \ref{theorem-polynomials}, $\Phi$ is an eigenfunction and hence the fibre $\Phi^{-1}(\{0\})$ is a compact minimal submanifold of $\SO {2n}$, by Theorem \ref{theorem-main}.
\end{example}

\section{The compact unitary group $\U n$}

In this section we construct minimal submanifolds of codimension two of the compact unitary group
$$\U{n}=\{z\in\GLC{n}\ |\ z\cdot \bar z^t=I_n,\}.$$ 
For this we use its standard $n$-dimensional representation on $\cn^n$ with 
$$z\mapsto
\begin{bmatrix}
z_{11} & \cdots & z_{1n} \\ 
\vdots & \ddots & \vdots \\
z_{n1} & \cdots & z_{nn}
\end{bmatrix}.$$ 
The Lie algebra $\u n$ of $\U n$ is the set of the complex  skew-Hermitian matrices
$$\u n=\{Z\in\glc n |\ Z+\bar Z^t=0\}$$ 
and for this we have the canonical orthonormal basis
$$ \B_{\u{n}} = \{ Y_{rs}, i X_{rs} \mid 1 \leq r  < s \leq n\} \cup \{ iD_t \mid t=1,\dots n\}.$$ 
The gradient $\nabla\phi$, of a complex-valued function $\phi:\U n\to\cn$, is an element of the complexified tangent bundle $T^\cn\U n$.  This satisfies 
$$\nabla\phi=\sum_{Z\in\B_{\u n}}Z(\phi)\cdot Z.$$

\begin{remark}\label{remark-z}
We now consider the functions $z_{j\alpha}: \U{n} \to \cn$ with $z_{j\alpha} : z \mapsto e_j \cdot z \cdot e_\alpha^t$. For any tangent vector $Z\in B_{\u{n}}$ we get
$$Z(z_{j\alpha}) : z \mapsto e_j \cdot z \cdot Z  \cdot e_\alpha^t = \sum_{k=1}^n z_{jk}\cdot Z_{k\alpha}.$$ 
As a direct consequence, we see that for the elements of the basis $\B_{\u n}$ we yield
\begin{eqnarray}
Y_{rs}(z_{j \alpha}) &=& \frac{1}{\sqrt{2}}(z_{jr}\cdot \delta_{\alpha s} -  {z_{js}} \cdot \delta_{\alpha r}),\\
iX_{rs}(z_{j\alpha }) &=& \frac{i}{\sqrt{2}}({z_{jr}} \cdot \delta_{\alpha s} + {z_{js}} \cdot \delta_{\alpha r }),\\
i D_{t}(z_{j \alpha}) &=& i\,z_{j \alpha} \cdot \delta_{\alpha t}.
\end{eqnarray}
\end{remark}
\smallskip

For the tension field $\tau$ and the conformality operator $\kappa$ on the unitary group $\U n$ we have the following result.

\begin{lemma}\cite{Gud-Sak-1},\cite{Gha-Gud-5}
Let $z_{j\alpha}:\U n\to\cn$ be the matrix elements of the standard representation of the unitary group $\U n$.  Then the tension field $\tau$ and the conformality operator $\kappa$  satisfy the following relations 
$$\tau(z_{j\alpha})=-\,n\cdot z_{j\alpha},\ \ 
\kappa(z_{j\alpha}, z_{k\beta})=-\,z_{j\beta}\,z_{k\alpha},$$
$$\tau(\bar z_{j\alpha})=-\,n\cdot \bar z_{j\alpha},\ \ 
\kappa(\bar z_{j\alpha}, \bar z_{k\beta})=-\,\bar z_{j\beta}\,\bar z_{k\alpha},$$
$$
\kappa(z_{j\alpha}, \bar z_{k\beta})=\delta_{jk}\cdot\delta_{\alpha\beta}.
$$
\end{lemma}

This leads to the following statement.

\begin{proposition}\cite{Gud-Sak-1}\label{proposition-U(n)-eigen}
Let $p$ be a non-zero element of $\cn^n$. Then the $n$-dimensional vector space
$$\E(p)=\{\phi_a:\U n\to\cn\ |\ \phi_a(z)=\trace (p^taz^t),\ a\in \cn^n\}$$ 
is an eigenfamily on $\U n$.
\end{proposition}

\begin{example}
Let the complex-valued function $\Phi:\U n\to\cn$ be defined by $\Phi:z\mapsto  z_{11}$.  Then the following coefficients of the gradient $\nabla\Phi$ satisfy
$$
iD_{1}(\Phi)=i\,z_{11},\ 
Y_{12}(\Phi)=-\frac {z_{12}}{\sqrt{2}},\ 
Y_{13}(\Phi)=-\frac {z_{13}}{\sqrt{2}},\ \dots\ ,\ 
Y_{1n}(\Phi)=-\frac {z_{1n}}{\sqrt{2}}.
$$
Since the first row $(z_{11},z_{12},\dots,z_{1n})$ can not vanish, at least one of these derivatives, and hence the gradient $\nabla\Phi$ of $\Phi$ is non-zero along the unitary group $\U n$.  This means that $\Phi$ induces a foliation of $\U n$ of codimension two.  It follows from Proposition \ref{proposition-U(n)-eigen} that $\Phi$ is an eigenfunction.  This implies that the fibre $\Phi^{-1}(\{0\})$ is a compact minimal  submanifold of $\U n$.
\end{example}

\begin{example}
For a positive integer $n\in \zn^+$, let $\Sigma_n$ be the following subset of $\cn^n$ given by 
$$\Sigma_n=\{(a_1,a_2,\dots ,a_n)\in\cn^n\, | \, a_k\neq 0\ \text{for}\  k=1,2,\dots,n\}.$$ 
For a positive integer $d\in\zn^+$ and $a\in\Sigma_n$, let $\Phi_a:\U n\to\cn$ be the complex-valued function with 
$$\Phi_a:z\mapsto a_1\cdot z_{11}^d+a_2\cdot z_{12}^d+\cdots +a_n\cdot z_{1n}^d.$$
Then it follows from Proposition \ref{proposition-U(n)-eigen} and Theorem \ref{theorem-polynomials} that the function $\Phi_a$ is an eigenfunction on $\U n$.  For a point $z\in\U n$ the following derivatives satisfy 
$$
iD_{1}(\Phi_a)=i\,a_1\, d \cdot z_{11}^d,\,\dots\, ,
iD_{n}(\Phi_a)=i\,a_n\, d \cdot z_{1n}^d.
$$
Since the first row $(z_{11},z_{12},\dots,z_{1n})$ can not vanish, at least one of these derivatives, and hence the gradient $\nabla\Phi_a$, does not vanish along the unitary group $\U n$.  This implies that the compact fibres of $\Phi_a:\U n\to\cn$ form a foliation of codimension two.  According to Theorem \ref{theorem-main} the fibre $\Phi_a^{-1}(\{0\})$ is a minimal submanifold of $\U n$.
It follows that we have constructed a complex $n$-dimensional familiy $\F_n$ of minimal submanifolds of $\U n$ of codimension two given by 
$$\F_n=\{\Phi_a^{-1}(\{0\})\,|\,a\in\Sigma_n\}.$$
\end{example}

\begin{example}
For $n\ge 3$, we define the complex-valued function $\Phi:\U {n}\to\cn$ by
$$\Phi:z\mapsto ( z_{11}\cdot z_{22}-z_{12}\cdot z_{21})=\det
\begin{bmatrix}
z_{11} & z_{12} \\
z_{21} & z_{22}
\end{bmatrix}.$$
According to Theorem 11.2 of \cite{Gha-Gud-5}, $\Phi$ is an eigenfunction on $\U n$.
For any element $Z$ in the Lie algebra ${\u n}$ we then have 
$$ Z(\Phi ): z \mapsto Z(z_{11}) \cdot z_{22} +z_{11} \cdot Z(z_{22}) - Z(z_{12}) \cdot z_{21}- z_{12} \cdot Z(z_{21}).$$
When applying this to our basis elements in $\B_{\u n}$ we then easily yield
$$iD_1(\Phi)(z)=iD_2(\Phi)(z)=i\,(z_{11}\cdot z_{22}-z_{12} \cdot z_{21})=i\,\det
\begin{bmatrix}
z_{11} & z_{12} \\
z_{21} & z_{22}
\end{bmatrix},$$
$$i X_{1s}(\Phi)(z)= i\,(z_{1s} \cdot z_{22 } - z_{12} \cdot z_{2s})=i\,\det
\begin{bmatrix}
z_{1s} & z_{12} \\
z_{2s} & z_{22}
\end{bmatrix}
,$$
$$i X_{2s}(\Phi)(z)=i\,(z_{11} \cdot z_{2s } - z_{1s} \cdot z_{21})=i\,\det
\begin{bmatrix}
z_{11} & z_{1s} \\
z_{21} & z_{2s}
\end{bmatrix}.$$
Since the first and the second rows of any matrix element $z\in\U n$ are linearly independent, we notice that the gradient $\nabla\Phi$ is non-vanishing along $\U n$.  This implies that the compact fibres of $\Phi:\U n\to\cn$ form a foliation on $\U n$ of codimension two.  According to Theorem \ref{theorem-main} the fibre $\Phi^{-1}(\{0\})$ is a compact minimal submanifold of $\U n$.
\end{example}

\section{The compact quaternionic unitary group $\Sp n$}

In this section we construct minimal submanifolds of codimension two of the compact quaternionic unitary group $\Sp n$.  The group $\Sp n$ is the intersection of the unitary group $\U{2n}$ and the standard representation of the quaternionic general linear group $\GLH n$ in $\cn^{2n\times 2n}$ given by
$$
(z+jw)\mapsto q=
\begin{bmatrix}
z & w \\
 -\bar w & \bar z
\end{bmatrix}.
$$ 
The Lie algebra $\sp n$ of $\Sp n$ satisfies
$$\sp{n}=\{\begin{bmatrix} Z & W
\\ -\bar W & \bar Z\end{bmatrix}\in\cn^{2n\times 2n}
\ |\ Z^*+Z=0,\ W^t-W=0\}.$$ 

For the indices $1\le r<s\le n$ and $1\le t\le n$ we now  introduce the following notation for the elements of the orthonormal basis $\B_{\sp n}$ of the Lie algebra $\sp n$ of the quaternionic unitary group $\Sp n$:
$$Y^a_{rs}=\frac 1{\sqrt 2}
\begin{bmatrix}
Y_{rs} & 0 \\
     0 & Y_{rs}
\end{bmatrix},\  
X^a_{rs}=\frac 1{\sqrt 2}
\begin{bmatrix}
iX_{rs} & 0 \\
      0 & -iX_{rs}
\end{bmatrix},$$
$$ X^b_{rs}=\frac 1{\sqrt 2}
\begin{bmatrix}
      0 & iX_{rs} \\
iX_{rs} & 0\end{bmatrix},\ 
X^c_{rs}=\frac 1{\sqrt 2}
\begin{bmatrix}
      0 & X_{rs} \\
-X_{rs} & 0
\end{bmatrix},$$
$$D^a_{t}=\frac 1{\sqrt 2}
\begin{bmatrix}
iD_{t} & 0 \\
     0 & -iD_{t}
\end{bmatrix},
D^b_{t}=\frac 1{\sqrt 2}
\begin{bmatrix}
     0 & iD_{t}  \\
iD_{t} & 0
\end{bmatrix},
D^c_{t}=\frac 1{\sqrt 2}
\begin{bmatrix}
     0 & D_{t}  \\
-D_{t} & 0
\end{bmatrix}.$$
\smallskip 

For $1\leq j,k,\alpha,\beta\leq n$, let us now consider the coordinate functions
$$z_{j\alpha}:q=z+jw \mapsto  e_j \cdot q \cdot e_\alpha^t,$$
$$w_{k \beta}:q=z+jw \mapsto   {e_{k}} \cdot q \cdot e_{n+\beta}^t.$$ 
For an arbitrary tangent vector $Z\in\B_{\sp n}$, we then have 
$$Z(z_{j\alpha}):q\mapsto e_j^t\cdot q\cdot Z\cdot e_\alpha^t=\sum_{\beta=1}^{2n} q_{j\beta} Z_{\beta \alpha}.$$
In particular, we obtain
$$ X^a_{rs}(z_{j \alpha}):q \mapsto  \frac{i}{\sqrt{2}} \sum_{\beta=1}^{n} q_{j \beta} (\delta_{r \beta} \delta_{s \alpha} + \delta_{s \beta}\delta_{r \alpha}) =  \frac{i}{\sqrt{2}}( z_{jr} \delta_{s \alpha} + z_{j s} \delta_{r \alpha}),$$
$$ X^b_{rs}(z_{j \alpha}):q \mapsto  \frac{i}{\sqrt{2}} \sum_{\beta=n+1}^{2n} q_{j \beta} (\delta_{r \beta} \delta_{s \alpha} + \delta_{s \beta}\delta_{r \alpha}) =  \frac{i}{\sqrt{2}}( w_{jr} \delta_{s \alpha} + w_{j s} \delta_{r \alpha}),$$
$$ D^a_{\alpha}(z_{j \alpha}):q \mapsto \frac{i}{\sqrt{2}}\sum_{\beta=1}^{n} q_{j \beta}(\delta_{\beta \alpha}) =  \frac{i}{\sqrt{2}} z_{j\alpha}.$$
$$ D^c_{\alpha}(z_{j \alpha}):q \mapsto \frac{1}{\sqrt{2}}\sum_{\beta=n+1}^{2n} q_{j \beta}(-\delta_{\beta \alpha}) =  \frac{-1}{\sqrt{2}} w_{j\alpha}.$$

For the quaternionic unitary group $\Sp n$ we have the following useful result, see \cite{Gud-Sak-1} and \cite{Gud-Mon-Rat-1}.

\begin{proposition}\label{proposition-Sp(n)-eigen}
Let $p$ be a non-zero element of $\cn^n$. Then the set
$$\E(p)=\{\phi_{ab}:\Sp n\to\cn\ |\ \phi_{ab}(g)=\trace (p^taz^t+p^tbw^t),\ a,b\in \cn^n\}$$  of complex-valued functions is an eigenfamily on $\Sp n$.
\end{proposition}

\begin{example}
On the quaternionic unitary group $\Sp n$ we define the complex-valued eigenfunction $\Phi:\Sp n\to\cn$ with
$$\Phi:q=z+jw\mapsto z_{11}.$$
Employing the above relations for the derivatives $Z(z_{11})$ with $Z\in\B_{\sp n}$, and noting that $s>1$, we then obtain 
$$ X^a_{rs}(z_{11})
=\frac{i}{\sqrt{2}}( z_{1r} \delta_{s1} + z_{1 s} \delta_{r1})
=\frac{i}{\sqrt{2}}(z_{1 s} \delta_{r1}),$$
$$ X^b_{rs}(z_{11})=
\frac{i}{\sqrt{2}}( w_{1r} \delta_{s1} + w_{1 s} \delta_{r1}))
=\frac{i}{\sqrt{2}}( w_{1 s} \delta_{r1}),$$
$$ D^a_{1}(z_{11}):q \mapsto \frac{i}{\sqrt{2}}\sum_{\beta=1}^{n} q_{1 \beta}(\delta_{\beta 1}) =  \frac{i}{\sqrt{2}} z_{11},$$
$$ D^c_{1}(z_{11}):q \mapsto \frac{1}{\sqrt{2}}\sum_{\beta=n+1}^{2n} q_{1 \beta}(-\delta_{\beta 1}) =  \frac{-1}{\sqrt{2}} w_{11}.$$
It is clear that the first row $(z_{11},\dots, z_{1n},w_{11},\dots, w_{1n})$ of any element $q\in\Sp n$ does not vanish.  This implies that the function $\Phi:\Sp n\to\cn$ is submersive along $\Sp n$.  Hence we obtain a foliation on $\Sp n$ of codimension two.  The leaves are compact and the fibre $\Phi^{-1}(\{0\})$ over $0\in\cn$ is minimal.
\end{example}

\section{Acknowledgements}

The authors are grateful to Fran Burstall for useful discussions on this work. 
They would also like to thank the referee for useful comments improving the presentation.


\end{document}